\documentclass{article}
\usepackage{graphicx}
\usepackage{amsthm, amsmath, amssymb, tikz}
\usepackage{enumitem}
\usepackage{xifthen}

\usepackage[margin=1.5in]{geometry} 

\usetikzlibrary{calc,shapes, backgrounds}

\newtheorem{lemma}{Lemma}
\newtheorem{theorem}{Theorem}
\newtheorem{corollary}{Corollary}

\newtheorem{proposition}{Proposition}

\newcommand{\ds}{\displaystyle}
\newcommand{\cJ}{\mathcal{J}}
\newcommand{\lp}{\left (}
\newcommand{\rp}{\right )}

\pgfdeclarelayer{edgelayer}
\pgfdeclarelayer{nodelayer}
\pgfsetlayers{edgelayer,nodelayer,main}

\title{A note on 1-guardable graphs in the cops and robber game}
\author{
Linyuan Lu
\thanks{University of South Carolina, Columbia, SC 29208,
({\tt lu@math.sc.edu}). This author was supported in part by NSF
grant DMS-1600811.} \and
Zhiyu Wang \thanks{University of South Carolina, Columbia, SC 29208,
({\tt zhiyuw@math.sc.edu}).} 
}
\begin{document}

\maketitle

\begin{abstract}
  In the cops and robber games played on a simple graph $G$, Aigner and Fromme's lemma states that one cop can guard a shortest path in the sense that the robber cannot enter this path without getting caught after finitely many steps. In this paper, we extend Aigner and Fromme's lemma to cover a larger family of graphs and give metric characterizations of these graphs. In particular, we show that a generalization of block graphs, namely vertebrate graphs, are 1-guardable. We use this result to give the cop number of some special class of multi-layer generalized Peterson graphs.
\end{abstract}

\section{Introduction}

The cop(s)-robber game is played by two players on a simple graph $G$ with $n$ vertices. The cop player first puts $c$ pebbles (called cops) on vertices of G. Then the robber player puts one pebble (called the robber) on some vertex of $G$. Two players move alternatively. The cop player can have each cop stay at its current position or move to one of its neighbors. Multiple cops are allowed to be placed on the same vertex. The robber player can have the robber stay at the same position or move to a neighboring vertex. If, after either player's turn, a cop and the robber are at the same vertex, the  robber is captured and the cop player wins. The cop number, $c(G)$, of a graph $G$ is the smallest positive integer $k$ such that $k$ cops suffice to capture the robber in finite number of moves. Let $c(n)$ be the maximum of $c(G)$ over all simple connected graphs with $n$ vertices.  In this game, both players are assumed to have complete information about the graphs and the positions of the pebbles.

The cops-robber game was first studied by Quilliot \cite{Quilliot} and independently considered by Nowakowski and Winkler who in \cite{Winkler} classified the {\em cop-win} graphs (i.e. graphs with cop number $1$). The cop number was introduced by Aigner and Fromme \cite{Aigner}. They also showed that if $G$ is planar, then $c(G)\leq 3$. 

Given a connected graph $G$, the distance between vertices $u$ and $v$ in $G$, denoted by $d_G(u,v)$, is the length of a shortest path in G connecting $u$ and $v$. An induced subgraph $H$ of $G$ is {\em isometric} (or {\em geodesic}) if for all vertices $u$ and $v$ of $H$, 
\[d_H (u,v) = d_G (u,v).\]
For example, a shortest path connecting two vertices is isometric. If $H$ is isometric in $G$, then the definition also implies that $H$ must be isometric in any subgraph of $G$ containing $H$ as induced subgraph.\\

For a fixed integer $k\geq 1$, $k$ cops can {\em guard} an induced subgraph $H$ in $G$ if, after finitely many moves, $k$ cops can move within $H$ in such a way that if the robber moves into $H$ at round $t$, then it will be captured at round $t+1$. Note that if one cop can guard a subgraph $H$ in $G$, then $H$ must be isometric in $G$. Suppose not, then there are two vertices $u,v \in V(H)$ such that $d_H(u,v) > d_G(u,v)$. The robber can travel between $u,v$ infinitely many times without being caught by the cop that guards $H$.\\

We say $H$ is {\em 1-guardable} if one cop can guard $H$ in any graph $G$ that contains $H$ as an isometric subgraph. Note that by the definition, if $H$ is 1-guardable, then $c(H) = 1$ (otherwise, the robber can stay in $H$ forever without being caught). 

Note that there exists {\em cop-win} graphs (graphs with cop number $1$) that are isometric in a larger graph $G$ but are not 1-guardable in $G$ (see Figure \ref{fig:notguardable}).

\begin{figure}[hbt]
\begin{center}
\tikzstyle{new}=[circle,fill=red,inner sep=1.5pt]
\tikzstyle{blue}=[circle,fill=blue, inner sep=1.5pt]
\tikzstyle{straight edge}=[]
\begin{tikzpicture}
	\begin{pgfonlayer}{nodelayer}
		\node [style=new] (0) at (-1, 1) {};
		\node [style=new] (1) at (0, 1) {};
		\node [style=new] (2) at (-0.5, 1.75) {};
		\node [style=new] (3) at (-0.5, 0.25) {};
		\node [style=new] (4) at (-1.5, 0.25) {};
		\node [style=new] (5) at (0.5, 0.25) {};
		\node [style=blue] (6) at (-0.5, 2.75) {};
	\end{pgfonlayer}
	\begin{pgfonlayer}{edgelayer}
		\draw [style=straight edge] (2) to (0);
		\draw [style=straight edge] (0) to (4);
		\draw [style=straight edge] (4) to (3);
		\draw [style=straight edge] (3) to (5);
		\draw [style=straight edge] (5) to (1);
		\draw [style=straight edge] (1) to (2);
		\draw [style=straight edge] (0) to (1);
		\draw [style=straight edge] (1) to (3);
		\draw [style=straight edge] (3) to (0);
		\draw [style=straight edge] (6) to (2);
		\draw [style=straight edge, bend right=60] (6) to (4);
		\draw [style=straight edge, bend right=60] (5) to (6);
	\end{pgfonlayer}
\end{tikzpicture}
\end{center}
\label{fig:notguardable}
\caption{The subgraph $H$ induced by the red vertices is isometric in
  the larger graph $G$. $H$ is cop-win, but not 1-guardable in $G$.}
\end{figure}
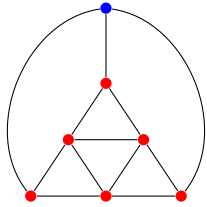

 The following lemma of Aigner and Fromme \cite{Aigner} on guarding isometric paths has found a number of applications. 

\begin{lemma}\label{AF}
\cite{Aigner}  An isometric path is 1-guardable.
\end{lemma}

Using this lemma and a greedy approach, Frankl \cite{Frankl} gave the first sub-linear upper bound on $c(n)$ by showing that $c(n) \leq  (1+o(1)) \lp\frac{n\log \log n}{\log n}\rp.$
This bound was improved later by Chiniforooshan \cite{Chiniforooshan} who showed that $c(n) = O\lp \frac{n}{\log n}\rp$. The best known upper bound is $c(n) = O\lp \frac{n}{2^{(1-o(1)) \sqrt{\log_2 (n)}}}\rp$, which was proven independently by Lu and Peng \cite{Lu},  Scott and Sudakov \cite{Sudakov}, as well as Frieze, Krivelevich and Loh \cite{Frieze}. However, this upper bound is still far away from the bound conjectured by Meyniel:  $c(n)=O(\sqrt{n})$.
The key ingredients of the results above all involve Lemma \ref{AF}. Inspired by the number of application of Aigner and Fromme's lemma, we are interested in finding a larger family of graphs that is 1-guardable. T. Ball, et al. \cite{Ball} extended  Aigner and Fromme's Lemma to isometric trees by showing that an isometric tree is $1$-guardable. They used this result to show that the cop number of generalized Peterson graph $GP(n,3)$ is less than or equal to $3$.

In this paper, we extend the result to larger families of isometric subgraphs
and give metric characterizations to these families
(see Theorem \ref{t6}, \ref{t7}, and \ref{t8}). The proofs of these results
are given in Section 3. An application
of these results is given in the last section.

\section{Preliminaries and results}

Before we state our theorems, we need some preliminary definitions. Let $G$ be a simple graph. From now on, we will always assume
that $G$ is connected.
For any vertex $x$ of $G$, define the neighborhood 
$N_G(x) = \{v \in V(G), vx \in E(G)\}$, and the closed neighborhood
$N_G[x] = N_G(x) \cup \{x\}$.
For any $c,x \in V(G)$ with $d(c,x) \geq 2$, define $N_G(c,x) = \{v \in N_G(c): d(c,x) = 1 + d(v,x)\}$, i.e. $N_G(c,x)$ is the collection of neighbors of $c$ in $G$ that lie on a shortest path between $c$ and $x$. We may omit the subscript G if it is clear from the context.\\

Consider the following three properties on the connected graph $G$:
\begin{description}

\item[Property $(P1)$:]
For any four vertices $c,c',x,y \in V(G)$
the largest two of the three distances $d(c,c') + d(x,y)$, $d(c,x) + d(c',y)$, and $d(c,y) + d(c',x)$ are always equal.

\item[Property $(P2)$:]
For any four vertices $c, x \in V(G)$, $c' \in N(c,x)$, $y \in V(G)$ such that $d(c,x) \geq 2$, $d(c',y)\geq 2$,
either $d(c,y)=d(c,c')+d(c',y)$ or $d(x,y)=d(x,c')+d(c',y)$.

\item[Property $(P3)$:]
There exists an induced subgraph $R$ of $G$ with $N[R] = G$ such that $\forall c \in V(R), x \in V(G)$ with $d(c,x) \geq 2$, there exists $c' \in N_R(c,x)$ such that $\forall y \in V(G)$ satisfying $\;d(c',y) \geq 2$, we have either $d(c,y) = d(c,c') + d(c',y) $ or $d(x,y) = d(x,c') + d(c',y)$.

\end{description}

\noindent{\bf Remark:} Property (P1) is better known as the \textit{four-point condition}. Property (P2) is the special case of Property (P1) with the restriction that $c' \in N(c,x)$ and $d(c',y) \geq 2$. Thus, (P1) implies (P2). Similarly, (P2) implies (P3) with the choice $R=G$.\\

A {\em block graph} (or {\em clique tree}) $G$ is a connected graph in which every {\em block} (i.e. maximal 2-connected subgraph) is complete. E. Howorka \cite{Howorka} showed a purely metric characterization of a block graph: a connected graph is a block graph if and only if it satisfies the four-point condition (i.e. Property (P1)). The {\em joints} of a block graph $G$, denoted as $\cJ(G)$, is the set of cut vertices of $G$.\\

In this paper, we define a larger family of graphs that extend block graphs (see examples in Figure \ref{fig:graphs}).\\

A connected graph $G$ is an {\em extended block graph} if it can be obtained from a block graph by blowing-up the cut vertices, i.e., replacing each cut vertex by a clique and connecting every vertex of the clique to every neighbor of the cut vertex. Each such clique is called a {\em joint block.} Each vertex in a joint block is called a {\em joint} of $G$ and we use $\cJ(G)$ to denote the set of all joints in an extended block graph $G$. Note that each extended block graph has a unique block graph associated with it by contracting each joint block to a single vertex.\\

A simple connected graph $G$ is a {\em vertebrate graph} if there is an induced subgraph $B$ such that 
	\begin{itemize}
		\item $B$ is an extended block graph. 
		\item For all $c \in V(B), x \in V(G)$ with $d(c,x) \geq 2$, there exists $c' \in N_B(c,x)$ such that $N_G[c'] \supseteq N_G[v]$ for all $v \in N_G(c,x)$.
	
	\end{itemize}

We call $B$ the {\em backbone} of G. Note that a block graph is an extended block graph, which, in turn, is a vertebrate graph. 

\begin{figure}[hbt]
\begin{center}
\tikzstyle{vertex}=[circle,fill=black,inner sep=1.5pt]
\tikzstyle{new}=[circle,fill=red,inner sep=1.5pt]
\tikzstyle{blue}=[circle,fill=blue, inner sep=1.5pt]
\tikzstyle{green}=[circle,fill=green, inner sep=1.5pt]

\tikzstyle{curly edge}=[]
\tikzstyle{straight edge}=[]

\begin{tikzpicture}
	\begin{pgfonlayer}{nodelayer}
		\node [style=blue] (0) at (-2, 1) {};
		\node [style=vertex] (1) at (-2, 0) {};
		\node [style=new] (2) at (-1.25, 0.5) {};
		\node [style=vertex] (3) at (-0.5, 1) {};
		\node [style=vertex] (4) at (-2.75, 1) {};
		\node [style=vertex] (5) at (-2.75, 0) {};
		\node [style=vertex] (6) at (-0.5, 0) {};
		\node [style=vertex] (7) at (-0.5, 0) {};
	\end{pgfonlayer}
	\begin{pgfonlayer}{edgelayer}
		\draw [style=curly edge] (0) to (2);
		\draw [style=curly edge] (2) to (3);
		\draw [style=curly edge] (2) to (1);
		\draw [style=curly edge] (1) to (5);
		\draw [style=curly edge] (0) to (4);
		\draw [style=curly edge] (0) to (1);
		\draw [style=straight edge] (6) to (2);
		\draw [style=straight edge] (3) to (6);
	\end{pgfonlayer}
\end{tikzpicture}
\hspace{1cm}% NO SPACE!
\begin{tikzpicture}
	\begin{pgfonlayer}{nodelayer}
		\node [style=new] (0) at (-2, 1) {};
		\node [style=new] (1) at (-2, 0) {};
		\node [style=vertex] (2) at (-1, 1) {};
		\node [style=vertex] (3) at (-1, 0) {};
		\node [style=blue] (4) at (-2.75, 1.25) {};
		\node [style=blue] (5) at (-3.5, 0.5) {};
		\node [style=vertex] (6) at (-2.75, -0.25) {};
		\node [style=vertex] (7) at (-3.5, 1.25) {};
		\node [style=vertex] (8) at (-3.5, -0.25) {};
	\end{pgfonlayer}
	\begin{pgfonlayer}{edgelayer}
		\draw [style=curly edge] (0) to (2);
		\draw [style=curly edge] (2) to (1);
		\draw [style=curly edge] (0) to (1);
		\draw [style=straight edge] (2) to (3);
		\draw [style=straight edge] (3) to (1);
		\draw [style=straight edge] (0) to (3);
		\draw [style=straight edge] (4) to (0);
		\draw [style=straight edge] (4) to (1);
		\draw [style=straight edge] (6) to (1);
		\draw [style=straight edge] (5) to (6);
		\draw [style=straight edge] (5) to (4);
		\draw [style=straight edge] (5) to (0);
		\draw [style=straight edge] (5) to (1);
		\draw [style=straight edge] (4) to (6);
		\draw [style=straight edge] (6) to (0);
		\draw [style=straight edge] (7) to (4);
		\draw [style=straight edge] (7) to (5);
		\draw [style=straight edge] (8) to (6);
	\end{pgfonlayer}
\end{tikzpicture}\hspace{1cm}% NO SPACE!
\begin{tikzpicture}
	\begin{pgfonlayer}{nodelayer}
		\node [style=new] (0) at (-2, 1) {};
		\node [style=new] (1) at (-2, 0) {};
		\node [style=vertex] (2) at (-1, 1) {};
		\node [style=vertex] (3) at (-1, 0) {};
		\node [style=blue] (4) at (-2.75, 1.25) {};
		\node [style=blue] (5) at (-3.5, 0.5) {};
		\node [style=vertex] (6) at (-2.75, -0.25) {};
		\node [style=vertex] (7) at (-3.5, 1.25) {};
		\node [style=vertex] (8) at (-3.5, -0.25) {};
		\node [style=green] (9) at (-4, 1.5) {};
		\node [style=green] (10) at (-0.5, 0.5) {};
		\node [style=green] (11) at (-4, 1) {};
	\end{pgfonlayer}
	\begin{pgfonlayer}{edgelayer}
		\draw [style=curly edge] (0) to (2);
		\draw [style=curly edge] (2) to (1);
		\draw [style=curly edge] (0) to (1);
		\draw [style=straight edge] (2) to (3);
		\draw [style=straight edge] (3) to (1);
		\draw [style=straight edge] (0) to (3);
		\draw [style=straight edge] (4) to (0);
		\draw [style=straight edge] (6) to (1);
		\draw [style=straight edge] (5) to (6);
		\draw [style=straight edge] (5) to (4);
		\draw [style=straight edge] (5) to (0);
		\draw [style=straight edge] (5) to (1);
		\draw [style=straight edge] (4) to (6);
		\draw [style=straight edge] (6) to (0);
		\draw [style=straight edge] (7) to (4);
		\draw [style=straight edge] (7) to (5);
		\draw [style=straight edge] (8) to (6);
		\draw [style=straight edge] (9) to (7);
		\draw [style=straight edge] (9) to (5);
		\draw [style=straight edge] (10) to (2);
		\draw [style=straight edge] (10) to (3);
		\draw [style=straight edge] (10) to (1);
		\draw [style=straight edge] (11) to (5);
		\draw [style=straight edge] (11) to (9);
		\draw [style=straight edge] (7) to (11);
		\draw [style=straight edge] (9) to (4);
		\draw [style=straight edge] (11) to (4);
	\end{pgfonlayer}
\end{tikzpicture}
\label{fig:graphs}
\caption{Example of block graph (left), extended block graph (middle) and vertebrate graph (right). The extended block graph is obtained from the block graph by replacing the blue and red cut vertices with a pair of adjacent blue and red vertices respectively. The green vertices in the vertebrate graph are the vertices not in the backbone.}
\end{center}
\end{figure}
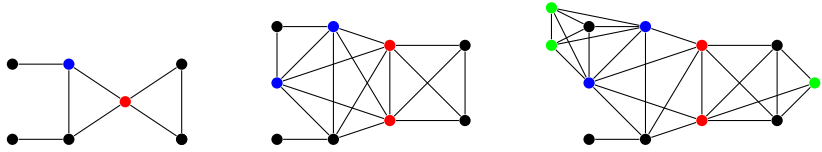

The main results of this paper are as follows:

\begin{theorem}\label{t6}
A simple connected graph $G$ satisfies $(P2)$ if and only if it is an extended block graph.
\end{theorem}

\begin{theorem}\label{t7}
A simple connected graph $G$ satisfies $(P3)$ if and only if it is a vertebrate graph. 
\end{theorem}

\begin{theorem}\label{t8}
Any vertebrate graph is 1-guardable.
\end{theorem}

The inclusive relations of these families of graphs are described in Figure \ref{fig:inclusions}.

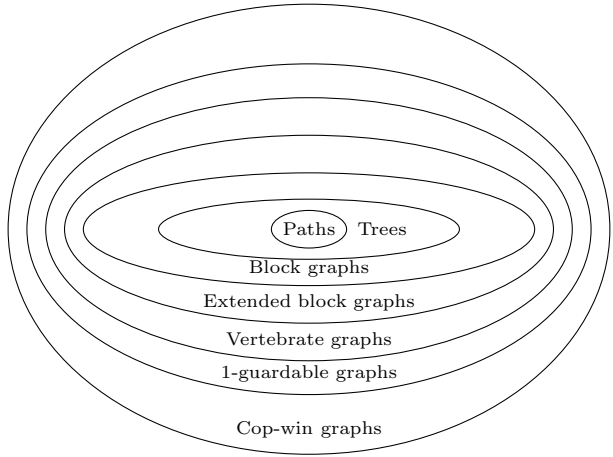
\begin{figure}[hbt]
\begin{center}
\begin{tikzpicture}
\node[ellipse, draw, minimum height=3cm,minimum width=8cm, text height = 4cm, label={[font=\scriptsize, anchor=southeast,above=1mm]270: Cop-win graphs}] (main) {};
\node[ellipse, draw, minimum height = 4.4cm, minimum width = 7.5cm,label={[font=\scriptsize, anchor=south,above=0.5mm]270:1-guardable graphs}] at (main.center)  (semi) {};
\node[ellipse, draw, minimum height = 3.5cm, minimum width = 7cm,label={[font=\scriptsize, anchor=south,above=0.5mm]270: Vertebrate graphs}] at (main.center) (active) {};
\node[ellipse, draw, minimum height = 2.5cm, minimum width = 6.5cm,label={[font=\scriptsize, anchor=south,above=0.5mm]270: Extended block graphs}] at (main.center) (active) {};
\node[ellipse, draw, minimum height = 1.5cm, minimum width = 6cm,label={[font=\scriptsize, anchor=south,above=0mm]270: Block graphs}] at (main.center) (active) {};
\node[ellipse, draw, minimum height = 0.8cm, minimum width = 4cm,label={[font=\scriptsize, anchor=east,above=4mm, right=.5mm]320: Trees}] at (main.center) (active) {};
\node[ellipse, draw, minimum height = 0.5cm, minimum width = 1cm,label={[font=\scriptsize, anchor=west,above=.5mm]270: Paths}] at (main.center) (active) {};
\end{tikzpicture}
\end{center}
\caption{Families of $1$-guardable graphs and cop-win graphs.}
\label{fig:inclusions}
\end{figure}

In Section 3, we will prove the main theorems. In Section 4, we use these results to give the cop number of some special class of multi-layer generalized Peterson graphs.

%Proof of Main Theorems ------------------------------------------------------------------------------------------------------------------------------

\section{Proof of theorems}

\subsection{Technical Lemmas}
 For ease of notation, for the rest of the paper, we use $N[u]$ to denote $N_G[u]$ and $N(c,x)$ to denote $N_G(c,x)$. Before we prove the main theorems, let us first establish some lemmas.
 
Let $G$ be a graph that satisfies $(P3)$ and $R$ be the subgraph of $G$ defined in $(P3)$ . Let $c \in V(R)$ and $x \in V(G)$. Define 
\[C(c,x) = \{u \in N_R(c,x) : \forall v \in N(c,x), N[u] \supseteq N[v]\}. \]

\begin{lemma}\label{l1}
For any $c \in V(R)$, $x \in V(G)$, and $c' \in N_R(c,x)$, if for any $y \in V(G)$ with $d(c',y) \geq 2$, either $d(c,y) = d(c',y) + 1$ or $d(x,y) = d(x,c') + d(c',y)$, then $c' \in C(c,x)$.
\end{lemma}

\begin{proof}
Let $v \in N(c,x)$. Then $v \in N(c')$ otherwise we get a contradiction to the (P3) condition (by setting $y=v$). Similarly, if there exists $y \in N[v]\backslash N[c']$, then $c,x,c',y$ also violates the (P3) condition. Hence $N[v] \subseteq N[c']$. It follows that $c' \in C(c,x)$.
\end{proof}

\begin{corollary}\label{c1}
Let $G$ be a graph that satisfies $(P3)$ and $R$ be the subgraph of $G$ defined in $(P3)$. Then for any $c \in V(R)$ and $x \in V(G)$ with $d(c,x) \geq 2$, we have $C(c,x) \neq \emptyset$.
\end{corollary}

\begin{corollary}\label{c2}
Let $G$ be a graph that satisfies $(P3)$ and $R$ be the subgraph of $G$ defined in $(P3)$.
Then for any $c\in V(R)$, $x\in V(G)$ with $d(c,x) \geq 2$, and any $c_1,c_2 \in C(c,x)$, we have $N[c_1] = N[c_2]$. 
\end{corollary}
\begin{proof}
By definition, since $c_1, c_2\in C(c,x)$, it follows that $N[c_1] \subseteq N[c_2]$ and $N[c_2] \subseteq N[c_1]$.

\end{proof}

Observe that if a graph $G$ satisfies $(P2)$, it also satisfies $(P3)$ with $R = G$. Hence we have the following corollaries.
\begin{corollary}\label{c3}
Suppose $G$ satisfies $(P2)$. Let $c,x \in V(G)$ with $d(c,x) \geq 2$. For all $c_1, c_2 \in N(c,x)$, we have
$N[c_1] = N[c_2]$.
\end{corollary}
\begin{proof}
For any $c_1, c_2 \in N(c,x)$, we have $c_1, c_2 \in C(c,x)$ by $(P2)$ and Lemma \ref{l1}. Hence by Corollary \ref{c2}, $N[c_1] = N[c_2]$.\\
\end{proof}

\begin{corollary}\label{c4}
Suppose $G$ satisfies $(P3)$. For any two vertices $c_1, c_2 \in V(G)$, if $N[c_1] \subseteq N[c_2]$, and $c_1 \in C(c,x)$ for some $c \in V(R), x \in V(G)$ with $d(c,x)\geq 2$, then $c_2 \in C(c,x)$. 
\end{corollary}
\begin{proof}
For any $v\in N(c,x)$, since $c_1 \in C(c,x)$, we have $N[v] \subseteq N[c_1] \subseteq N[c_2]$. 
Hence $c_2 \in C(c,x)$.
\end{proof}

\begin{corollary}\label{c5}
Suppose $G$ satisfies $(P2)$. For any $c_1, c_2 \in V(G)$, if $N[c_1] \subseteq N[c_2]$, and $c_1 \in N(c,x)$ for some $c, x \in V(G)$ with $d(c,x) \geq 2$, then $c_2 \in C(c,x)$. 
\end{corollary}

\begin{lemma}\label{l2}
Suppose $G$ satisfies $(P3)$ and R be the subgraph of $G$ defined in $(P3)$. Let $\cJ = \{C(c,x): c \in V(R),x \in V(G)\}$. If $J_1, J_2 \in \cJ$, then $J_1, J_2$ are either identical or disjoint.
\begin{proof}
Let $J_1 = C(c_1,x_1)$ and $J_2 = C(c_2, x_2)$.
If $J_1$ and $J_2$ are not disjoint,
we can find some $c' \in J_1 \cap J_2$.
For any $c'' \in C(c_1, x_1)$, $N[c'] = N[c'']$ by Corollary \ref{c2}.
It follows from Corollary \ref{c4},  that $c'' \in C(c_2,x_2)$.
Hence $J_1 \subseteq J_2$.
Similarly, we can show that $J_2\subseteq J_1$.
Hence $J_1 = J_2$.
\end{proof}
\end{lemma}

\begin{lemma} \label{l3}
Suppose $G$ is a vertebrate graph and $B$ is the backbone of G. For all $u,v \in V(G)$ with $d(u,v) \geq 2$, there exists a shortest path between $u$ and $v$ whose internal vertices all lie in $B$.
\begin{proof}
Note that if either of $u,v$ is in $B$, then by the second condition in the definition of vertebrate graph, we are done. Otherwise, $u,v \in G\backslash B$. We proceed by contradiction. Suppose that there exists $u,v \in V(G)\backslash V(B)$ with $d(u,v)\geq 2$ such that all shortest paths between $u,v$ have an internal vertex not in $B$. This implies that if $P$ is a shortest path between $u$ and $v$, then all internal vertices of $P$ are not in $B$ (otherwise there exists a shortest path between $u$ and $v$ whose internal vertices all lie in $B$). 

Pick such $u,v$ with the minimum distance. It easily follows that $d(u,v) = 2$. Suppose $uwv$ is the shortest path between $u$ and $v$ with $u,w,v \in V(G)\backslash V(B)$.
Let $w' \in B$ be an arbitrary neighbor of $w$ in $B$ ($w'$ exists since $N[R]=G$). If $w'$ is adjacent to both $u$ and $v$, then we are done. Otherwise, WLOG, $d(w',v)  = 2$, i.e. there exists $v' \in N_B(w',v)$ such that $N[v']\supseteq N[w]$ since $w \in N(w',v)$. In particular, $v'$ is adjacent to $u$. Hence $uv'v$ is a shortest path between $u$ and $v$ whose internal vertices all lie in $B$, which gives us a contradiction by the choice of $u$ and $v$. 

\end{proof}
\end{lemma}

\subsection{\bf Proof of Theorem \ref{t6}.}

\begin{proposition}
If $G$ is connected and satisfies (P2), then $G$ is an extended block graph.
\end{proposition}

\begin{proof}
	
Suppose $G$ is connected and $G$ satisfies $(P2)$.
Let $u,u' \in V(G)$ such that $d(u,u')$ equals to the diameter $d$ of $G$. Fix $u$ from now on. Define \[N_k(u)= \{v\in V(G): d(u,v) = k\}.\] 
Consider one of the components of the graph induced by $N_d(u)$. Call this component C. Let $N(C) = \{v\in N_{d-1}(u): vw\in E(G) \;\text{for some } w \in C\}$.\\\\
{\bf Claim:} $N(C) \cup C$ induces a complete graph.

\begin{proof}[Proof of Claim]
Let $v,w \in V(C)$ and $vw\in E(G)$. Let $v' \in N(v,u)$, i.e. $v'$ is the neighbor of $v$ in the shortest $vu$ path. Similarly, $w' \in N(w,u)$. Suppose $v'\neq w'$. We will show that $v'w, w'v, w'v' \in E(G)$.

\begin{figure}[hbt]
\begin{center}

\tikzstyle{straight edge}=[]
\tikzstyle{vertex}=[circle,fill=black,inner sep=1.5pt]

\begin{tikzpicture}[scale=0.8]
	\begin{pgfonlayer}{nodelayer}
		\node [style=vertex, label=above:{$v$}] (0) at (3.25, 1) {};
		\node [style=vertex, label=below:{$w$}]  (1) at (3.25, -1) {};
		\node [style=vertex, label=above:{$v'$}]  (2) at (2.5, 0.75) {};
		\node [style=vertex, label=below:{$w'$}] (3) at (2.5, -0.75) {};
		\node [style=vertex, label=left:{$u$}] (4) at (0, 0) {};
	\end{pgfonlayer}
	\begin{pgfonlayer}{edgelayer}
		\draw [style=straight edge] (2) to (1);
		\draw [style=straight edge] (3) to (0);
		\draw [style=straight edge] (0) to (1);
		\draw [style=straight edge] (2) to (3);
		\draw [style=straight edge] (4) to (0);
		\draw [style=straight edge] (4) to (1);
	\end{pgfonlayer}
\end{tikzpicture}
\end{center}
\label{fig:extendedblockgraph}
\caption{$N(C)\cup C$ induces a complete graph}
\end{figure}
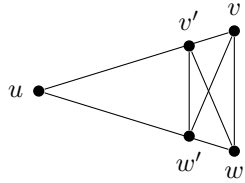
We claim that $v'$ is adjacent to $w$. Suppose not, i.e. $d(v',w) \geq 2$. Then by the property $(P2)$, either $d(v,w) = d(v',w)+1$ or $d(u,w) = d(u,v') + d(v', w)$.\\
However, 
\[ d(v,w ) = 1 < d(v',w) + 1,\]
\[ d(u,w)  = d < d+1 \leq  d(u,v') + d(v',w). \]
Contradiction. Thus $v'$ is adjacent to w. Similarly, $w'$ is also adjacent to $v$.
Now $v', w' \in N(v,u)$, by Corollary \ref{c3}, $v'w' \in E(G)$.
To show that $N(C) \cup C$ is a complete graph, it suffices to show that C induces a complete graph.
Suppose not. Then there must be a path P of length at least 2 in C. Hence there exists a subpath $v_1, v_2, v_3 \in V(C)$ such that $d(v_1, v_3) \geq 2$.
Let $v_2' \in N(v_2, u)$. By our previous argument, $v_2'$ is also adjacent to $v_1$ and $v_3$. Thus $v_2', v_2 \in N(v_1, v_3)$. It follows that by Corollary \ref{c3}, $N[v_2'] = N[v_2]$, but this is impossible since $v_2'$ is on the shortest path between $u$ and $v_2$. Hence by contradiction, C induces  a complete graph, and thus $N(C) \cup C$ induces a complete graph.
\end{proof}

Now we will prove $G$ is an extended block graph by inducting on $|V(G)|$. 
The base case is clearly true since a single vertex is an extended block graph.
Now let's consider $G' = G\backslash C$. Since $N(C)$ is complete graph, we have \[\forall x,y \in V(G'), d_{G'}(x,y) = d_G (x,y).\]
Since $G$ satisfies $(P2)$, it follows that $G'$ also satisfies $(P2)$. By induction, $G'$ is an extended block graph. By Lemma \ref{l2}, $N(C)$ either is a joint block itself in $G'$ or is disjoint with any joint block in $G'$. Since $C$ it a complete graph, it follows that $G = G' \cup C$ is an extended block graph.
\end{proof}

\begin{proposition}
Suppose $G$ is an extended block graph, then $G$ satisfies (P2).
\end{proposition}
\begin{proof}
Let $G$ be an extended block graph. Observe that for $c,x,c',y \in V(G)$ with $c' \in N(c,x)$ and $d(c',y) \geq 2$, if no two of the four vertices are in the same joint block, then the block graph $B \subseteq G$ associated with $G$ containing $c,x,c',y$ preserves the distance function, i.e. \[\forall s,t \in V(B), d_G(s,t) = d_B(s,t).\]
Since $B$ is a block graph, it satisfies $(P1)$, thus $(P2)$. Hence $c,x,c',y$ satisfies $(P2)$.
Therefore, it suffices to show $G$ satisfies $(P2)$ for $c,x,c',y \in V(G)$ when there exists two of the four vertices that are in the same joint block. Note that if two vertices $s, t$ are in the same joint block, then $N[s] = N[t]$. It's easy to check that no two of the four vertices   $c,x,c',y $ can be in the same joint block except the pair $x,y$. But if $x,y$ are in the same joint block, it follows that $d(c,y) = d(c,c') + d(c',y)$, in which case we are done.
\end{proof} 

%-------------------------Proof of Theorem 7--------------------------------------------------------------

\subsection{\bf Proof of Theorem \ref{t7}}

\begin{proposition}
If $G$ satisfies (P3), then $G$ is a vertebrate graph.
\end{proposition}
%---------------------------(P3) implies Vertebrate Graph-----------------------------------------------
\begin{proof}
Suppose $G$ satisfies $(P3)$. Let $S =\displaystyle\bigcup_{\substack{c\in V(R), x \in V(G)\\ d(c,x) \geq 2}} C(c,x)$ where $R$ is defined in $(P3)$ and let $B = G[S]$. \\

{\bf Claim:} $B$ is an extended block graph. 

By Theorem \ref{t6}, it suffices to show that $B$ satisfies $(P2)$. Observe that for all $x,y \in V(B)$, $d_B(x,y) = d_G(x,y)$.
Let $c,x \in V(B), c' \in N_B(c,x)$. By (P3), there exists $c'' \in N_R(c,x)$ such that $\forall y \in V(G)$ with $d(c'',y)\geq 2$, either $d(c,y) = d(c'',y) + 1$ or $d(x,y) = d(x,c'') + d(c'',y)$.
We claim that $N[c'] = N[c'']$:
By Lemma \ref{l1}, $c'' \in C(c,x) \subseteq V(B)$. By the definition of $C(c,x)$, it follows that $N[c'] \subseteq N[c'']$. Now since $c' \in V(B)$, it follows that $c' \in C(c_0,x_0)$ for some $c_0\in V(R), x_0\in V(G)$. By Corollary \ref{c4}, it follows that $c'' \in C(c_0,x_0)$. Now since $c',c'' \in C(c_0,x_0)$, by Corollary \ref{c2} we have $N[c'] = N[c'']$.
By our choice of $c''$, it then follows that  $\forall y \in V(B)$ with $d(c',y)\geq 2$, either $d(c,y) = d(c',y) + 1$ or $d(x,y) = d(x,c') + d(c',y)$. Hence $B$ is indeed an extended block graph.\\

It remains to show that for all $c \in V(B), x \in V(G)$ with $d(c,x) \geq 2$, there exists $c' \in N_B(c,x)$ such that $N[c'] \supseteq N[v]$ for all $v \in N(c,x)$. 
Indeed, since $G$ satisfies (P3), there exists $c' \in N_R(c,x)$ such that for all $y \in V(G)$ with $d(c',y) \geq 2$, we have either $d(c,y) = d(c,c') + d(c',y) $ or $d(x,y) = d(x,c') + d(c',y)$. Note that by Lemma \ref{l1}, $c' \in C(c,x)\in V(B)$. We claim $c'$ has the above property.
Suppose not, there exists $v \in N(c,x)$ and $y' \in N[v]\backslash N[c']$. Note that $d(c',y') \geq 2$. However, 
\[d(c,y') \leq d(c,v) + d(v,y') < d(c,c') + d(c',y'),\]
\[d(x,y') \leq d(x,v)  + d(v,y') < d(x,c') + d(c',y').\]
which contradicts the (P3) condition.\\

\end{proof}

%-------------------------Vertebrate Graph implies (P3)---------------------------------------------------------------------------
\begin{proposition}
If $G$ is a vertebrate graph, then $G$ satisfies $(P3)$.
\end{proposition}
\begin{proof}
Let $G$ be a vertebrate graph. Let $B'$ be the maximal backbone (which is an extended block graph) of a vertebrate graph $G$. We want to show that $G$ satisfies $(P3)$.

Since $B'$ is chosen to be maximal, every vertex in $V(G)\backslash V(B')$ is adjacent to at least some joint of $B'$. Let $B = \{c' \in N_{B'}(c, x): c\in V(B'), x\in V(G)$ and $N_G[c'] \supseteq N_G[v]$ for all $ v \in N_G(c,x)\}$. It's easy to see that every vertex in $V(G)$ is adjacent to some vertex in $B$. Moreover, $B$ is also an extended block graph. We will use $B$ as the backbone of $G$ from now on.

Let $R = B$. Clearly $N[R]=N[B]=G$. Suppose $c \in V(R),x \in V(G)$ with $d(c,x)\geq 2$. Since $G$ is a vertebrate graph, there exists $c' \in N_B(c,x)$ such that $N[c'] \supseteq N[v]$ for all $v \in N_G(c,x)$. We want to show that for all $y \in V(G)$ with $d(c',y)\geq 2$, either $d(c,y) = d(c',y)+1$ or $d(x,y) = d(x,c') + d(c',y)$.
Let $G' = G[V(B) \cup \{x,y\}]$. Since $G$ is a vertebrate graph and $G'$ is connected, it follows that $G'$ is also a vertebrate graph. Moreover, by lemma 3, for all $u,v \in V(B) \cup \{x,y\}$, $d_{G}(u,v) = d_{G'}(u,v)$. Hence it suffices to prove our claim in $G'$ from now on.

Let $P$ be the shortest path between $c$ and $x$ in $G'$. Note that $c'$ is a joint vertex of $B$. Let $J_{c'}$ be the joint block of B containing $c'$. Let $J = J_{c'} \cup N_{G'}(c,x)$. Observe that by our choice of $c'$, $N[c'] \supseteq N[v]$ for all $v \in J$.

\begin{figure}[hbt]
\begin{center}

\tikzstyle{straight edge}=[]
\tikzstyle{vertex}=[circle,fill=black,inner sep=1.5pt]

\begin{tikzpicture}[scale=0.8]
	\begin{pgfonlayer}{nodelayer}
		\node [style=vertex,label=below:{$c$}] (0) at (-4, 0) {};
		\node [style=vertex,label=below:{$c'$}] (1) at (-2.5, 0) {};
		\node [style=vertex,label=below:{$x'$}] (2) at (0, 0) {};
		\node [style=vertex,label=below:{$x$}] (3) at (1, 0) {};
		\node [style=vertex,label=above:{$x''$}] (4) at (0, 2) {};
		\node [style=vertex,label=above:{$c_1$}] (5) at (-3, 0.5) {};
		\node [style=vertex,label=above:{$c_2$}] (6) at (-2, 1) {};
		\node [style=vertex] (7) at (-1, 1.5) {};
		\node [style=vertex,label=left:{$z$}] (8) at (0, 1) {};
		\node [style=vertex] (9) at (-1.25, 0) {};
	\end{pgfonlayer}
	\begin{pgfonlayer}{edgelayer}
		\draw [style=straight edge] (0) to (3);
		\draw [style=straight edge] (0) to (5);
		\draw [style=straight edge] (5) to (6);
		\draw [style=straight edge] (6) to (7);
		\draw [style=straight edge] (7) to (4);
		\draw [style=straight edge] (4) to (3);
		\draw [style=straight edge] (4) to (8);
		\draw [style=straight edge] (8) to (2);
		\draw [style=straight edge] (8) to (3);
	\end{pgfonlayer}
\end{tikzpicture}

\end{center}
\label{fig:separator}
\caption{$c$ and $x$ are in different components of $G'\backslash J$}
\end{figure}
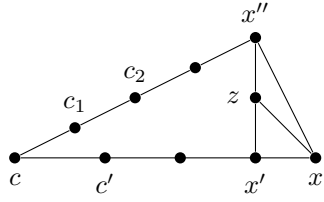

{\noindent\bf Claim:} $c$ and $x$ are in different components of $G' \backslash J$. 
\begin{proof}[Proof of Claim]
Suppose not, i.e. there exists another path $Q = c c_1 c_2 ... x''x$ from $c$ to $x$ such that $Q$ does not pass through $J$ ($c_2$ exists since length of $Q$ is larger than length of $P$). Take $Q$ to be the shortest such path. Since $G'$ is a vertebrate graph and $d(c',y) \geq 2$, we can assume that the internal vertices of $Q$ all lie in $B$. 

Suppose $x' = N_P(x,c)$ and $x'' = N_Q(x,c)$. Observe that the joint blocks of all vertices in $P$ are all distinct. Same holds for $Q$. Let $J_c$ and $J_{c_1}$ be the joint block of $c$ and $c_1$ respectively. 
\begin{description} 
\item[Case 1:] $c_1$ is not adjacent to $c'$. Then c is on the shortest path between $c_1$ and $c'$. Since $B$ is an extended block graph, it follows that $x''$, $x'$ should be in different components of $B\backslash J_c$. However, since $x$ is adjacent to both $x'$ and $x''$, by lemma 3, it follows that $d_B(x',x'') \leq 2$. In particular, it follows that $d_B(x',x'') = 2$ and $c \in N_B(x'',x')$. By our choice of $B$, it follows that $x\in N_{G'}(c)$ which contradicts that $d(c,x) \geq 2$. Hence by contradiction, $c$ and $x$ are in different components of $G' \backslash J$.

\item[Case 2:] $c_1$ is adjacent to $c'$. Then $c_2$ cannot be adjacent to $c'$. Otherwise by our choice of $B$ and $c_1$, for all $v\in N_B(c,c_2)$, $N_{G}[v] \subseteq N_{G}[c_1]$. Hence if $c'$ is adjacent to $c_2$, it will follow that $N_G[c'] \subseteq N_G[c_1]$ and $c_1 \in N_{G'}(c,x) \subset J$, which contradicts that $Q$ does not pass through $J$. Thus by contradiction, $c_2$ is not adjacent to $c'$, i.e. $c_1$ is on the shortest path between $c_2$ and $c'$. Hence $x'$ and $x''$ are in different components of $G'\backslash J_{c_1}$. Similar to Case 1, $d_B(x',x'') \leq 2$. Thus it follows that $c_1 \in N_B(x'',x')$ and $x \in N_{G'}(c)$. But this will imply that $c_1 \in N_{G'}(c,x) \subset J$, contradicting that $Q$ does not pass through $J$ again.  
\end{description}
Hence in both cases, by contradiction, $c$ and $x$ are in different components of $G' \backslash J$. 
\end{proof}

Now $y \notin J$ since $d(c',y)\geq 2$. Hence $y$ is either not in the component of $G'\backslash J$ containing $c$ or not in the component containing $x$. If $y$ is not in the component of $G'\backslash J$ containing $c$, then the shortest path between $y$ and $c$ must pass through $J$. But for all $v \in J, N[v] \subseteq N[c']$. Hence it follows that 
\[d(c,y) = d(c,c') + d(c',y)\]
Similarly, if $y$ is not in the component of $G'\backslash J$ containing $x$, then 
\[d(x,y) = d(x,c')+d(c',y)\]

\end{proof}

\subsection{\bf Proof of Theorem \ref{t8}.}
Let $G$ be a simple connected graph and $H$ be an isometric subgraph of G. Assume at least two cops are available (The extra cop is only needed for finitely many steps).
We claim that if $H$ is a vertebrate graph, then it is 1-guardable.

Let $R \subseteq H$ be the subgraph of $H$ defined in $(P3)$.
  Move a cop to a vertex of $R$. Assume that at this point, the cop is at vertex $c\in V(R)$ and the robber is at vertex $r\in V(G)$. Define
a function $f\colon V(H)\times V(G)\to \mathbb{Z}$ as follows:
\begin{equation}
  \label{eq:frc}
f(c,r)=\min_{x\in V(H)} \{d(r,x)-d(c,x)\}.  
\end{equation}

By the triangle inequality, we have 
\begin{equation}
\label{eq:tri}
d(r,c)\geq f(r,c)\geq -d(r,c).
\end{equation}

{\noindent\bf Claim:} One cop can maintain $f(r,c)$ non-decreasing.
That is, if $f(r,c) \le 0$ and the robber moves from $r$ to a new position
$r'$, then the cop can move to a position $c'\in V(R)$ such that
\begin{equation}
  \label{eq:inc}
f(r',c')\geq f(r,c).  
\end{equation}
The claim implies that the robber can never enter $H$ after the time $f(r,c)\geq 0$.

\begin{proof}Assume $f(r,c)=-k$ for some nonnegative integer $k$. Suppose that the robber moves from $r$ to $r'$. (Since there is more than one cop, the robber has to move eventually.) 
%	
%If $r'$ is in $H$, the cop simply moves one step toward $r'$ along any shortest path from $c$ to $r'$ in $H$. Now we assume $r'$ is not in $H$. 
The cop guarding $H$ will use the following strategy:
\begin{description}
	\item Case 1: If $k > 0$ and $f(r',c) > f(r,c)$, then the cop stays at its current position.
	\item Case 2: If $k = 0$ and $f(r',c) \geq f(r,c)$, then the cop stays at its current position.
	\item Case 3: Otherwise, by (P3), there exists $c' \in N(c,x)$ such that $\forall y \in V(H),\;d(c',y) \geq 2$, we have either $d(c,y) = d(c',y) +1 $ or $d(x,y) = d(x,c') + d(c',y)$. Move the cop from $c$ to $c'$.

\end{description}

We will verify that under this strategy, $f(r',c') \geq -k$.
We assume $k>0$. The case when $k=0$ will follow the same line. We may also assume that $f(r',c) \leq f(r,c)$ since otherwise the cop stays at its current position and we have $f(r',c') = f(r',c) > f(r,c) =-k$ . 
By the triangle inequality, $|f(r',c)-f(r,c)|\leq d(r',r)=1$, Thus, we must have
\begin{equation}
  \label{eq:r1c}
 -k-1\leq f(r',c)\leq -k.  
\end{equation}
%In particular, $r'$ is not in $H$.
Suppose $x \in V(H)$ is the vertex that minimizes $d(r',x)-d(c,x)$. Then we have
\begin{equation}
  \label{eq:x}
-k-1 \leq d(r',x)-d(c,x) \leq -k.  
\end{equation}

Now we will verify inequality \eqref{eq:inc}.
Suppose not, i.e., $f(r',c')\leq -k-1$. It follows that there is a vertex $y\in V(H)$ such that
\begin{equation}
  \label{eq:y}
d(r',y)-d(c',y)\leq -k-1.  
\end{equation}
Combining Equations \eqref{eq:x} and \eqref{eq:y}, we get
\begin{equation}
  \label{eq:xy}
  [d(r',x)+d(r',y)]-[d(c,x)+d(c',y)]\leq -2k-1.
\end{equation}

% Equivalently, we have
% \begin{equation}
%   \label{eq:cxc'y}
%   d(c,x)+d(c',y)\geq d(r',x)+d(r',y)+2k+2.
% \end{equation}

Note $d(c,c')=1$ and $d(c',y)\geq d(r',y)+k+1\geq 2$.
Hence by Property (P3), either $d(x,y)=d(x,c')+d(c',y)$ or
$d(c,y)=d(c',y)+ d(c',c)$.

There are two cases:
\begin{description} 
\item[Case 1:] $d(x,y)=d(x,c')+d(c',y)$.
Since $d(c,x) = d(c',x) + 1$, we then have 
\begin{equation}
  \label{eq:path}
  d(c,x)+d(c',y)=d(x,y)+1.
\end{equation}
Combining \eqref{eq:xy} and \eqref{eq:path}, we get
\begin{equation}
  [d(r',x)+d(r',y)]-d(x,y)\leq -2k<0.
\end{equation}
which contradicts the triangle inequality.
\item[Case 2: ] $d(c,y) = d(c,c') + d(c',y) = 1+ d(c',y)$.
Now by triangle inequality,
\begin{equation}
  \label{eq:f}
	 d(r,y) \leq d(r,r') + d(r',y) = 1 + d(r',y). 
\end{equation}

It follows that from \eqref{eq:f} and \eqref{eq:y} that
	\begin{align*}
		d(r,y) - d(c,y) &= d(r,y) - d(c',y) -1 \\
				  &\leq 1+d(r',y) -d(c',y)-1\\
				&\leq -k -1
	\end{align*}
which contradicts that $f(r,c) = -k$.
\end{description}
This finishes the proof of the claim. 
\end{proof}

{\noindent\bf Claim:} With the help of one additional cop, $f(r,c)$ cannot remain negative forever. Note that we only need the additional cop for finitely many rounds.
\begin{proof}
Suppose $B$ is the backbone of the vertebrate graph $H$. Since $N[B] = H$, it follows that at any point, in order to guard $H$, the cop $c$ only needs to be in $V(B)$. 
We will prove the claim by contradiction. Let $c_t, r_t$ denote the vertices occupied by the cop $c$ and the robber $r$ at round $t$, respectively. At each round,  the robber moves first.

Suppose the claim is false, i.e. there exists a $t_0$ such that after round $t_0$, the robber $r$ can move in a way that no matter how $c$ moves after $r$ makes its move, $f(r_t,c_t) = -k$ for all $t \geq t_0$ where $k$ is strictly positive.
Let $p$ be the additional cop that helps $c$. Note that every time the robber $r$ stays at it current vertex in a round, $p$ can reduce its distance to $r$ by 1.
Hence there exists $t_0'$ such that after round $t_0'$, $r$ has to move every round.
Assume $c$ moves in the same manner described in the previous claim.  Note that since by our assumption $f(r_t,c_t) =-k$ for all $t\geq t_0$,  it must happen that $f(r_{t+1}, c_t) \leq -k$ (otherwise the cop can just stay at $c_t$). By our algorithm (choice of $c'$) in the above claim, it follows that the cop will then always move to a different position.

\begin{figure}[hbt]
\begin{center}

\tikzstyle{straight edge}=[]
\tikzstyle{vertex}=[circle,fill=black,inner sep=1.5pt]
\begin{tikzpicture}[scale=0.6]
	\begin{pgfonlayer}{nodelayer}
		\node [style=vertex,label=below:{$c_{t-1}$}] (0) at (-4, 0) {};
		\node [style=vertex,label=below:{$x$}] (1) at (1, 0) {};
		\node [style=vertex,label=above:{$r_{t}$}] (2) at (1, 3) {};
		\node [style=vertex,label=above:{$r_{t-1}$}] (3) at (2, 3) {};
		\node [style=vertex,label=below:{$c_{t}$}] (4) at (-3, 0) {};
		\node [style=vertex,label=right:{$c_{t+1}$}] (5) at (-3, 1) {};
		\node [style=vertex,label=above:{$y$}] (6) at (-3, 3) {};
		\node [style=vertex,label=above:{$r_{t+1}$}] (9) at (0, 3) {};
	\end{pgfonlayer}
	\begin{pgfonlayer}{edgelayer}
		\draw [style=straight edge] (0) to (4);
		\draw [style=straight edge] (4) to (1);
		\draw [style=straight edge] (1) to (2);
		\draw [style=straight edge] (2) to (3);
		\draw [style=straight edge] (2) to (9);
		\draw [style=straight edge] (4) to (5);
		\draw [style=straight edge] (6) to (9);
		\draw [style=straight edge] (6) to (5);
		\draw [style=straight edge] (0) to (5);
	\end{pgfonlayer}
\end{tikzpicture}

\end{center}
\label{fig:backmove}
\caption{$f(r,c)$ cannot remain negative forever.}
\end{figure}
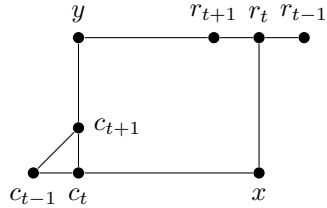
Since $H$ is finite, there will exists a moment $t$ such that $c_{t-1}, c_t, c_{t+1}$ are in the same maximal clique of $B$. Let $r_{t-1}, r_t, r_{t+1}$ be the vertices occupied by the robber accordingly. Since $f(r_{t},c_{t}) = -k$, there exists $x \in V(H)$ such that 
\begin{equation}\label{eq:ct}
d(c_t,x) = d(r_t,x) + k.
\end{equation}
Similarly, $f(r_{t+1}, c_{t+1}) = -k$, thus there exists $y \in V(H)$ such that 
\begin{equation}\label{eq:ctplus}
d(c_{t+1},y) = d(r_{t+1},y) + k.
\end{equation}
Note that by the moving strategy of the cop, $c_t, c_{t+1}$ are not in the same joint block. Now since $c_{t-1}, c_t, c_{t+1}$ are joints, pairwise adjacent, and $c_t \in N(c_{t-1},x)$, it must happen that 
\begin{equation}\label{eq:moveback}
d(c_{t+1},x) = d(c_t,x) + 1.
\end{equation}
Since $c_{t+1} \in N(c_t,y)$, clearly $d(y,c_t) \geq 2$. Thus by (P3) condition, either 
	\begin{equation}\label{eq:short1}
		d(c_{t-1},y) = 1 + d(c_t,y).
	\end{equation}
or 
	\begin{equation}\label{eq:short2} 
		d(x,y) = d(x,c_t) + d(c_t,y).
	\end{equation}

But $d(c_{t-1},y) \leq  d(c_{t-1},c_{t+1}) + d(c_{t+1},y) < 1 + d(c_t,y)$, thus \eqref{eq:short1} fails. Hence \eqref{eq:short2} must hold. Now by \eqref{eq:moveback}, we have 
\begin{align*} 
d(c_{t+1},y) +1+ d(c_t,x) & = d(x,y)\\
				 & \leq d(x, r_t) + d(r_t, r_{t+1}) + d(r_{t+1},y)\\
				&\leq d(c_t,x)-k + 1 + d(c_{t+1},y) -k\\
\end{align*}

Hence $2k\leq 0$, which contradicts that $k > 0$. By contradiction, $f(r,c)$ cannot remain negative forever.
\end{proof}
Thus after finite number of moves, $f(r,c)$ will turn non-negative and remain non-negative thereafter. It follows that $H$ is $1$-guardable.

\tikzstyle{every node} = [circle, draw, fill=white!50, inner sep = 1pt, minimum width= 2pt]

\newcommand{\drawblock}[3]{%number of blocks, number of layers, k value,
	\pgfmathsetmacro\unitangle{360/(#1)}
	\pgfmathsetmacro\offset{1}
	\pgfmathsetmacro\edge{1/2}
	\pgfmathsetmacro\layer{(#2)-1};
	\pgfmathsetmacro\layerint{int((#2)-1)};
	\pgfmathsetmacro\outer{int(\layer)-1};
	\pgfmathsetmacro\a{int(#1)-1};
	\coordinate (z) at (0,0) {};
	
	\foreach \j in {0,...,\a}{

		%draw the block vertices
		node (y{\j,0}) at (\j * \unitangle: \offset) {};
		\pgfmathsetmacro\rangle{90 + \unitangle * (\j + 0.5)};

		%draw the blocks of vertices with a fixed index
		\foreach \i in {0,...,\outer}{
			\pgfmathsetmacro\next{int(\i+1)};
			\ifthenelse{\j < #3}
				{\node[color=red] (y{\j,0}) at (\j * \unitangle: \offset) {};
				 \draw[red] (y{\j,\i}) --++ (\j * \unitangle:  \edge) node[color=red] (y{\j,\next}) {};}
				{\node (y{\j,0}) at (\j * \unitangle: \offset) {};
				\draw (y{\j,\i}) --++ (\j * \unitangle:  \edge) node (y{\j,\next}) {};}
		}

		%draw the curve that makes each block a complete graph
		\pgfmathsetmacro\d{#2+1};	
		\ifthenelse{\j < #3}{
			\draw[red] (y{\j,0}) to [out = 20+ \j * \unitangle, in=160+ \j * \unitangle] (\j* \unitangle: \d * \edge) node[red] () {};
		}{\draw (y{\j,0}) to [out = 20+ \j * \unitangle, in=160+ \j * \unitangle] (\j* \unitangle: \d * \edge) node () {};}

	}
	\foreach \j in {0,...,\a}{
		\pgfmathsetmacro\next{int(Mod(int(\j+1),#1))};
		\pgfmathsetmacro\minus{#3-1}
		\ifthenelse{\j < \minus}
			{\draw[red] (y{\j,\layerint}) --++ (y{\next,\layerint});}
			{\draw (y{\j,\layerint}) --++ (y{\next,\layerint});}

		\pgfmathsetmacro\jump{int(Mod(int(\j+#3),#1))};
		\foreach \i in {0,...,\outer}{
			\draw (y{\j,\i}) --++ (y{\jump,\i});
		}
	}
	 
}

\section{Multi-layer generalized Peterson Graph}
In this section, we will use the results in this paper to bound the cop number of some multi-layer generalized Peterson graphs.
A {\em generalized Peterson graph}, $GP(n,k)$ is the undirected graph having vertex set $A \cup B$, where $A = \{a_0, \ldots, a_{n-1}\}$ and $B = \{b_0, \ldots, b_{n-1}\}$, and having the following edges: $(a_i,a_{i+1}), (a_i, b_i)$ and $(b_i, b_{i+k})$ for each $i = 0, \ldots, n-1$, where indices are to be read modulo $n$.
T. Ball, et al. \cite{Ball} proved that the cop number of $GP(n,3)$ is less than or equal to three.\\

We generalize their idea to a larger structure. The {\em multi-layer generalized Peterson graph}, denoted as $MGP(n,k,t)$, is the undirected graph having vertex set $\ds\bigcup_{i=0}^{t} V_i$ where $V_j = {v_0^j, ..., v_{n-1}^j}$ and have the following edges: $(v_i^0 v_{i+1}^0), (v_i^j v_{i+k}^j)$ for each $i = 0, \dots, n-1$ where indices are to be read modulo $n$ and $j \in [t]$. Moreover, for a fixed index $i \in [n-1] \cup \{0\}$, the set $B_i = \{v_i^k, 0\leq k \leq t\}$ forms a complete graph.

\begin{center}
\begin{tikzpicture}
\drawblock{14}{3}{3}
\end{tikzpicture}

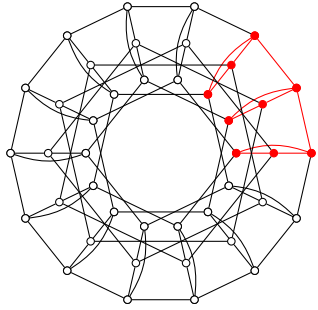
Figure 3: The multi-layer generalized Peterson graph $G(14,3,2)$.
\end{center}

\begin{center}
\begin{tikzpicture}
\drawblock{14}{3}{2}
\end{tikzpicture}

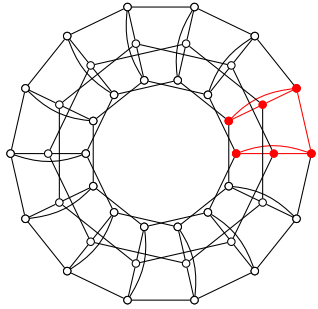
Figure 4: The multi-layer generalized Peterson graph $G(14,2,2)$.
\end{center}

Note that Peterson graph is exactly $G(5,2,1)$ and generalized Peterson graph $GP(n,k) = MGP(n,k,1)$. 

\begin{theorem}
The cop number of $MGP(n,2,t)$ is equal to $3$ for all $t \in \mathbb{N}$.
\end{theorem}

\begin{theorem}
The cop number of $MGP(n,3,t)$ is equal to $3$ for all $t \in \mathbb{N}$.
\end{theorem}

\begin{proof}
We will prove Theorem $9$ and $10$ together.

Consider $H =G[\ds\cup_{i=1}^k B_i]$(highlighted in red) where $k=2$ and $3$ for Theorem $9$ and $10$ respectively. Observe that $H$ is isometric and is a block graph. By Theorem 6, it follows that $H$ is 1-guardable, i.e. after finite number of rounds, it the robber enters $H$, it will be immediately captured in the next round by the cop guarding H. Suppose $c_0$ guards $H$. Consider $G\backslash H$, the remaining two cops $c_1, c_2$ will play the lifted strategy used in \cite{Ball}.

For ease of notation, if a cop $c$ lies on the vertex $v_i^j$, we say that the cop is on a vertex of layer $j$ and index $i$.
Then the strategy of the cops is as follows:
\begin{enumerate}[label=(\alph*)]
\item Move $c_1$ and $c_2$ both to $V_0$ (layer 0). Increase the index of $c_1$ by 1, and decrease the index of $c_2$ by $1$ every round. Since the robber can only increase or decrease its index by $0,1$ or $k$, there will be a moment that one of the two cops, say $c_1$, occupies a vertex whose index is congruent modulo $k$ to the index of the robber vertex. Note that the robber can change his index (mod $k$) by at most 1 every round. 
\item The robber $R$ will make its move and $c_1$ will then make a move such that it is on the same layer as $R$ while maintaining the same index (mod $k$) as R. In particular, if $R$ is already on layer $0$ in step $(a)$ and moves along layer $0$ in step (b), then $c_1$ will move in the same direction as $R$ in layer $0$. Otherwise, if $R$ is on other layers and makes any move (or pass), its index mod $k$ will not change. Hence $c_1$ can simply move to the same layer with $R$ while maintaining the same index (mod $k$) as R.
\item On subsequent rounds, $c_1$ will move in a way so that $c_1$ is always on the same layer and index (mod $k$) and whenever possible, move towards the robber. 
\item The other cop, $c_2$ always moves in $V_0$ towards R.
\end{enumerate}

Note that with $c_1, c_2$ moving in this way, after finite number of rounds, they can always force the robber to move in one direction (in terms of index) and eventually enters $H$ in which case the robber will be captured by $c_0$ in the next round. Thus, $c(MGP(n,3,t)) \leq 3$. Similarly, $c(MGP(n,2,t)) \leq 3$.\\

Note that $c(MGP(n,3,t))$ and $c(MGP(n,2,t))$ are both greater than 2. This is because if the robber stays any layer $V_i$ where $i\neq 0$, it can move to its neighbor on the left or right in the same layer, or move to its neighbor in another layer with the same index. Hence if there is a winning strategy for the robber with 2 cops, the 2 cops have to threaten all of the three possible places that the robber can escape, which is impossible.
Hence we can conclude that $c(MGP(n,3,t)) = c(MGP(n,2,t)) = 3$.

\end{proof}

\end{document}